\documentclass{article}
\usepackage[utf8]{inputenc}
\usepackage{cite}
\usepackage{hyperref}

\title{Klein-four connections and the Casson invariant\\ for non-trivial admissible $U(2)$ bundles}
\author{Christopher Scaduto \& Matthew Stoffregen}
\date{}

\usepackage[a4paper, total={6in, 8.5in}]{geometry}
\usepackage{graphicx}
\usepackage[dvipsnames]{xcolor}
\usepackage{amssymb}
\usepackage{tikz-cd}
\usepackage{titling}
\usepackage{amsthm}
\usepackage{amscd}
\usepackage{caption}
\usepackage{amsmath}
\usepackage{MnSymbol}
\usepackage{tikz}
\usepackage{dsfont}
\usepackage{sectsty}

\sectionfont{\fontsize{12}{15}\selectfont}

\newcommand{\R}{\mathbb{R}}

\newcommand{\Z}{\mathbb{Z}}
\newcommand{\Q}{\mathbb{Q}}

\newcommand{\su}{\mathfrak{s}\mathfrak{u}}

\newcommand{\adE}{\su(E)}

\newcommand{\CE}{\mathcal{C}_E}

\newcommand{\CadE}{\mathcal{C}_{\adE}}

\newcommand{\tr}{\text{Tr}}
\newcommand{\G}{\mathcal{G}}

\newcommand{\GE}{\mathcal{G}_E}

\newcommand{\GadE}{\mathcal{G}_{\adE}}
 
\newcommand{\BE}{\mathcal{B}_E}
\newcommand{\BadE}{\mathcal{B}_{\adE}}

\newcommand{\M}{\mathcal{M}} 

\newcommand{\Mpflat}{\M}

\newcommand{\ad}{\text{ad}}

\newcommand{\func}{f}

\newtheorem{prop}{Proposition}
\newtheorem{theorem}{Theorem}
\newtheorem{lemma}{Lemma}

\newtheorem{defn}{Definition}
\newtheorem{corollary}{Corollary}
\newtheorem{remark}{Remark}

\begin{document}

\maketitle 

\begin{abstract}
Given a rank 2 hermitian bundle over a 3-manifold that is non-trivial admissible in the sense of Floer, one defines its Casson invariant as half the signed count of its projectively flat connections, suitably perturbed. We show that the 2-divisibility of this integer invariant is controlled in part by a formula involving the mod 2 cohomology ring of the 3-manifold. This formula counts flat connections on the induced adjoint bundle with Klein-four holonomy.
\end{abstract}

\section{Introduction}
Let $E$ be a $U(2)$ bundle over a closed, oriented and connected 3-manifold $Y$ with the property that $w_2(E)$ has no torsion lifts to $H^2(Y;\Z)$. Following Floer \cite{floer}, we call such bundles \emph{non-trivial admissible}. Floer defined the instanton homology $I_\ast(Y,E)$, which is an abelian group that is $\Z_2$-graded. Define $\lambda(Y,E)$ to be half the euler characteristic of the instanton homology:
\[
	\lambda(Y,E)=\frac{1}{2}\chi\left[ I_\ast(Y,E)\right].
\]
This number is a signed count of suitably perturbed projectively flat connections on $E$. It is well-known that $\lambda(Y,E)$ is an integer. Define the subset of triples
\[
	V_Y = \{ \{a,b,c\} \subset H^1(Y;\Z_2): \; a + b + c =0\}.
\]
This set is naturally in correspondence with the set of subspaces of the $\Z_2$-vector space $H^1(Y;\Z_2)$ of dimension at most two. Write $b_1(2)$ for the $\Z_2$-dimension of $H_1(Y;\Z_2)$. Define for any given $x\in H^2(Y;\Z_2)$ the following non-negative integer:
\[
    v_Y(x)=\left|\Big\{\{a,b,c\}\in V_Y:\; ab+bc+ac = x \Big\}\right|.
\]
For the case in which $x=w_2(E)$ we simply write $v_Y(E)$.\\
\vspace{.25cm}

\begin{theorem}\label{thm:main}
	Suppose $E$ is a non-trivial admissible $U(2)$ bundle over a closed, oriented, connected 3-manifold $Y$ with $b_1(2)\geq 3$. Then $\lambda(Y,E)$ is divisible by $2^{b_1(2)-3}$. Furthermore, we have
\begin{equation}
	2^{3-b_1(2)}\lambda(Y,E) \equiv v_Y(E) \mod 2.\label{eq:cong}
\end{equation}
If $b_1(2)=2$, this congruence also holds, implying that $v_Y(E)$ is even. If $b_1(2)=1$, then the integer $v_Y(E)$ is zero. In these two cases $v_Y(E)$ {\emph{(mod 2)}} yields no information about $\lambda(Y,E)$.
\end{theorem}
\vspace{.25cm}

\noindent Note that $Y$ supports a non-trivial admissible bundle if and only if $b_1(Y)\geq 1$, where $b_1(Y)$ denotes the rank of $H_1(Y;\Z)$. In general we have $b_1(2)\geq b_1(Y)$, with strict inequality if and only if $H_1(Y;\Z)$ has 2-torsion. Theorem \ref{thm:main} and its proof are generalizations of a rather simple idea due to Ruberman and Saveliev \cite{rs1}. Their result is the case of Theorem \ref{thm:main} when $H_1(Y;\Z)$ is free abelian of rank 3, i.e., when $Y$ is a homology 3-torus. To obtain their statement, one identifies $v_Y(E)$ with the triple cup product modulo 2, which for a homology 3-torus is a simple computation. (More generally, see the corollary below.) Our adaptation of Ruberman and Saveliev's argument is summarized, modulo perturbations, as follows. 

The invariant $\lambda(Y,E)$ is one half of a signed count of projectively flat connections on the bundle $E$. There is an action of $H^1(Y;\Z_2)$ on this set of connections, and the quotient is identified with flat connections on the adjoint $SO(3)$ bundle induced by $E$. The only possible stabilizers of this action are $\{1\}$, $\Z_2$, and $V_4$, the Klein-four group isomorphic to $\Z_2\times \Z_2$. Further, the connections with stabilizer $V_4$ are flat connections with holonomy group $V_4$. The number $v_Y(E)$ is the number of connections on the induced $SO(3)$ bundle with holonomy $V_4$, up to gauge equivalence. The proof of Theorem \ref{thm:main} follows from counting the $H^1(Y;\Z_2)$-orbits with stabilizer $V_4$.\\

\vspace{.25cm}
\noindent {\bf{Vanishing conditions, and relation to Lescop's invariant.}}
The quantity $v_Y(E)$ (mod 2) of congruence (\ref{eq:cong}) is often, but not always, equal to zero. The parity also turns out to be independent of our choice of non-trivial admissible bundle $E$. To state the result,
\[
	k(Y) := \dim_{\Z_2} \{ a\in H^1(Y;\Z_2): \; a^2 =0 \} = \dim_{\Z_2}\text{ker}(\beta^1).
\]
Here $\beta^1$ is the Bockstein homomorphism defined on $H^1(Y;\Z_2)$ associated to the coefficient exact sequence $0 \to \Z_2 \to \Z_4 \to \Z_2\to 0$. As is well-known, $\beta^1(a)=a^2$. We note that if $H_1(Y;\Z)$ is written as a direct sum of prime-power order cyclic summands and copies of $\Z$, then $k(Y)$ is just the number of $\Z_{2^k}$ summands with $k>1$, plus the number of $\Z$ summands. In particular, $k(Y) \geq b_1(Y)$.\\

\begin{theorem}
	Let $Y$ be a closed, oriented and connected 3-manifold with $k(Y)\geq 1$. Let $x\in H^2(Y;\Z_2)$ be any element that is not a cup-square. Then $v_Y(x)$ {\emph{(mod 2)}} is independent of the choice of such $x$. If furthermore $k(Y) \geq 4$ then we have $v_Y(x)\equiv 0$ {\emph{(mod 2)}}.\label{thm:2}\\
\end{theorem}

\noindent Note that the statement holds for a larger class of elements $x\in H^2(Y;\Z_2)$ than those just coming from admissible bundles. The conditions are best understood through examples. The simplest interesting examples are certain surgeries on the Borromean rings, see Figure \ref{fig:borromean}. These examples are chosen such that $b_1(Y)=0$ and $b_1(2)=3$.\\

\noindent {\textbf{Example 1.}} Consider the 3-manifold $Y$ obtained by doing $(2,2,4)$ surgery on the Borromean rings. Such a manifold has first homology group isomorphic to $\Z_{2}\oplus \Z_{2}\oplus \Z_{4}$. Then $k(Y)=1$. The rank 3 vector space $H^1(Y;\Z_2)$ has basis $a,b,c$ with 
\[
	c^2=0, \quad a^2=bc, \quad b^2 =ac,
\]
and for which $ab,bc,ac$ form a basis for $H^2(Y;\Z_2)$. Now, $ab$ is not a square, as are not $ab+bc$, $ab+ac$ or $ab+ac+bc$. All four of these elements have $v_Y(x)=1\in \Z$. On the other hand, all other elements in $H^2(Y;\Z_2)$ have $v_Y(x)\in \{0,2,4\}$. This illustrates the necessity of the non-square condition on $x$.\\

\noindent {\textbf{Example 2.}} Next, consider $(2,4,4)$ surgery on the Borromean rings. The $\Z_2$-cohomology ring is much the same as before, except now $b^2=0$, and $k(Y)=2$. All non-zero $x\in H^2(Y;\Z_2)$ have $v_Y(x)$ odd. In fact, if $x\neq 0$, then $v_Y(x)=1$, while  $v_Y(a^2)=5$ and $v_Y(0)=4$. Here $a^2$ is a cup-square, but does not have a different parity from the other non-zero elements.\\

\noindent {\textbf{Example 3.}} Finally, the (4,4,4) surgery on the Borromean rings has the same $\Z_2$-cohomology ring as that of the 3-torus. Here $k(Y)=3$. In this case $v_Y(x)=1$ for $x\neq 0$, all of which are not squares, while $v_Y(0)=8$.\\

\begin{figure}[t]
\centering
\includegraphics[scale=.55]{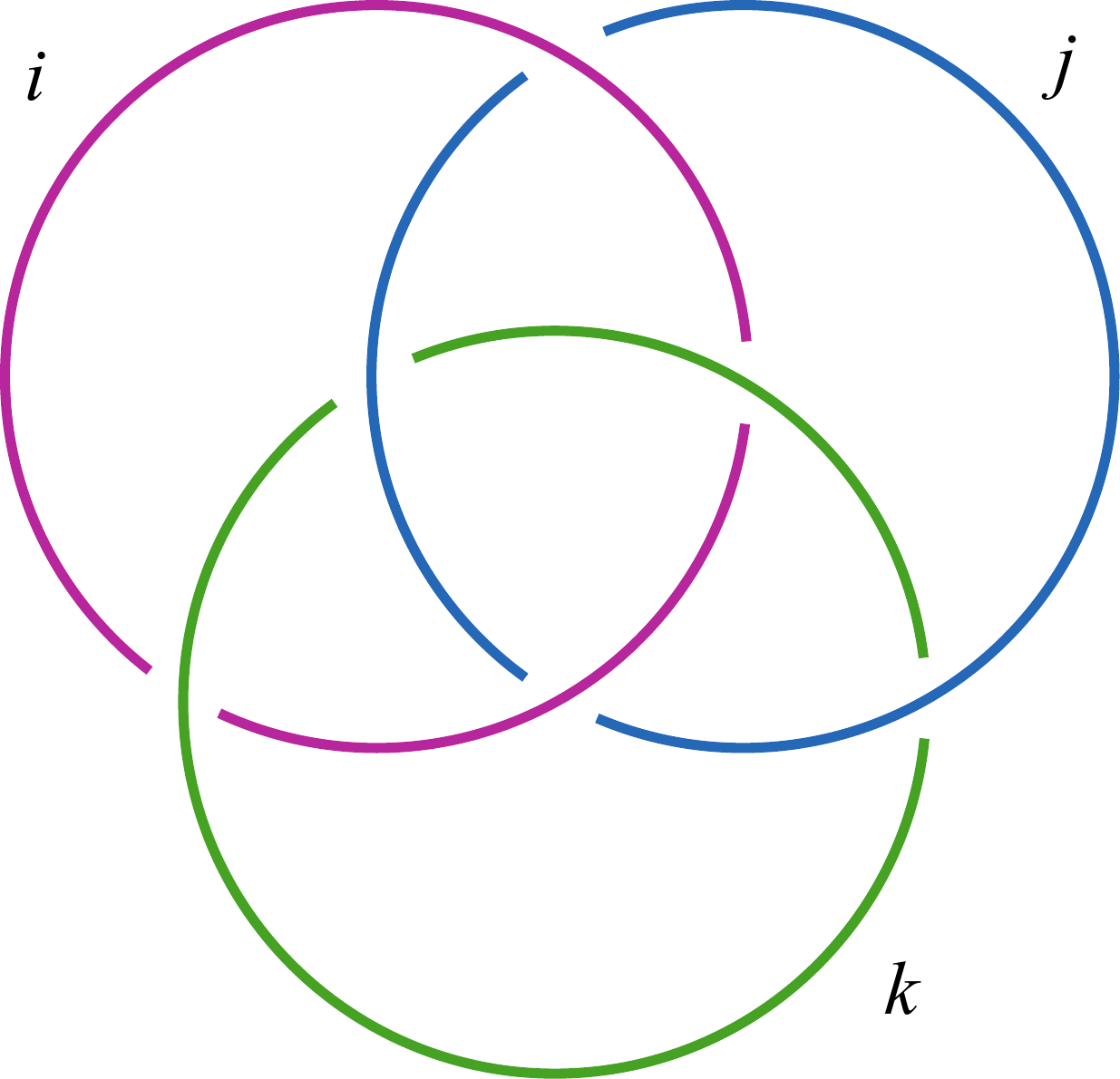}
\caption[]{{\small{
Surgery on the Borromean rings with framings $(j,k,l)$ on the three components. When $j,k,l$ are either 0 or various powers of 2, these surgeries yield non-vanishing examples of the congruence in Theorem \ref{thm:main}, in which $v_Y(E)\equiv 1$ (mod 2) and $k(Y)=1,2,3$.
}}}\label{fig:borromean}
\end{figure}

\noindent To make use of Theorem \ref{thm:main}, one can replace the 4-framings in the above three examples by $0$-framings, to get manifolds with the same $\Z_2$-cohomology rings but $b_1(Y)>0$, ensuring that they support non-trivial admissible bundles.

In what follows, we describe how to deduce Theorem \ref{thm:2} using Theorem \ref{thm:main} and related results of Poudel \cite{poudel} and Turaev \cite{turaev}. By Poudel \cite{poudel}, the Casson invariant $\lambda(Y,E)$ may be identified with Lescop's invariant of \cite{lescop}, slightly modified. The proof utilizes Floer's exact triangle for instanton homology and Dehn surgery techniques \`{a} la Lescop \cite{lescop}. As a result, the parity of $v_Y(E)$ is independent of $E$, the choice of non-trivial admissible bundle. After some substitutions, the congruences resulting from Theorem \ref{thm:main} and \cite{poudel} may be summarized as follows.\\

\begin{corollary}
Suppose $x\in H^2(Y;\Z_2)$ has no torsion lifts to $H^2(Y;\Z)$. Then modulo 2 we have\label{cor:main}
\begin{equation}
v_Y(x) \equiv 
\begin{cases}
	2^{2-b_1(2)}\Delta_Y''(1), & b_1(Y)=1\\
	2^{3-b_1(2)}\#(\gamma\cap F), & b_1(Y) = 2\\
	2^{3-b_1(2)}N\cdot(a\cup b\cup c)[Y],\;\;\; & b_1(Y) =3\\
	0, & b_1(Y)\geq 4
\end{cases}\label{eq:congs}
\end{equation}
where $N$ is the cardinality of ${\emph{\text{Tor}}}H_1(Y;\Z)$ and other terms are defined below. In particular, if $b_1(Y) = 3$ and $H_1(Y;\Z)$ has an element of order 4, then $v_Y(x)\equiv 0 \mod 2$.\\
\end{corollary}
\vspace{.20cm}

\noindent The right hand sides are defined as follows. First, for $b_1(Y)=1$, $\Delta_Y(t)$ is the Alexander polynomial of $Y$, normalized so that $\Delta_Y(1)=1$. If $Y$ is $0$-surgery on a knot $K$ in an integral homology 3-sphere $\Sigma$, then $\Delta_Y(t)$ is just the Alexander polynomial $\Delta_{K\subset \Sigma}(t)$. Next, suppose $b_1(Y)=2$. Take two oriented surfaces in $Y$ that generate $H_2(Y;\Q)$. Let $\gamma$ be their intersection, and $\gamma'$ the curve parallel to $\gamma$ that induces the trivialization of the tubular neighborhood of $\gamma$ given by the surfaces. Then $N\cdot \gamma'$ has a Seifert surface $F$ in $Y$, and $\#(\gamma\cap F)$ is the count of intersection points, in general position. Finally, in the $b_1(Y)=3$ case,  the triple $a,b,c$ generates $H^1(Y;\Z)$ up to torsion, and $[Y]$ is the fundamental class of $Y$.

The vanishing implications of the corollary above look rather similar to that of Theorem \ref{thm:2}, except that the role of $k(Y)$ is weakened to that of $b_1(Y)$. In other words, the role of counting summands of the form $\Z$ and $\Z_{2^k}$ for $k>1$ is replaced by that of just counting $\Z$ summands. From the perspective of the $\Z_2$-cohomology ring, these kinds of summands are all the same. With this thought in mind, it is a rather straightforward task to establish Theorem \ref{thm:2} from Corollary \ref{cor:main} using realization results for the $\Z_2$-cohomology structure of 3-manifolds due to Turaev. See Section \ref{sec:thm2}. We remark that, a posteriori, the divisibility properties of the quantities listed in the above corollary should imply Theorem \ref{thm:2}. However, the authors prefer to mostly argue with the $\Z_2$-cohomology ring structure, in line with the definition of $v_Y(x)$.\\

\vspace{.25cm}
\noindent {\bf{Some more examples.}} For any abelian group $H$ containing an element of order 4 or $\infty$, there is a 3-manifold $Y$ with $H_1(Y;\Z)$ isomorphic to $H$ and $v_Y(x)=0$, in which $x$ is any non-cup-square. For this, just consider integer-framed surgeries on unlinks. Note also that the integer $v_Y(x)$ is stable under connect sums with $\R\mathbb{P}^3$, which increases $b_1(2)$ by 1 while fixing $k(Y)$. This operation, applied to the three Borromean surgeries examples above, gives examples where $v_Y(x)\equiv 1\mod 2$ for any pair $b_1(2)$, $k(Y)$ such that $b_1(2)\geq 3$ and $k(Y)\in \{1,2,3\}$. In fact, it is straightforward to produce non-vanishing examples with $H_1(Y;\Z)$ any isomorphism class of abelian group with those same two constraints. We also have examples from Seifert-fibered spaces, with orientable base orbifold:\\

\begin{prop}
	Let $Y$ be a Seifert-fibered space with Seifert invariants $(g,b,(\alpha_1,\beta_1),\ldots,(\alpha_r,\beta_r))$, where $g$ is the genus of the base orbifold. Suppose $x\in H^2(Y;\Z_2)$ is not a square. Then $v_Y(x) \equiv 1$ {\emph{(mod 2)}} if and only if $g=1$, all $\alpha_i$ are odd, and $b + \sum \beta_i  \equiv 0$ {\emph{(mod 2)}}.\\ \label{prop:main}
\end{prop}

\noindent We note that such Seifert fibered spaces have $b_1(Y)\in \{2,3\}$ and $b_1(2)=3$. Included in this list is of course the 3-torus. This proposition is easily proven using the description of the mod 2 cohomology ring of a Seifert fibered space given in \cite{adhsz}. See Section \ref{sec:ex}. 

We mention that the Seifert-fibered spaces considered here for genus $g=0$ are double branched covers of Montesinos links. However, by the above proposition, the relevant invariant $v_Y(x)$ in these cases is always even. In Section \ref{sec:ex} we give an example of a double branched cover for which Theorem \ref{thm:main} has a non-vanishing congruence. \\

\vspace{.25cm}

\noindent {\textbf{Discussion.}} The integers $v_Y(x)$, and not just their parities, are interesting in the context of $SO(3)$ gauge theory. Indeed, as is evident in the sequel, the $V_4$-connection classes counted by $v_Y(E)$ are persistent (unmoved) under a large class of perturbations. As such, they form a distinguished set of generators in the instanton Floer chain complex for the pair $(Y,E)$, defined using any such perturbation. Klein-four connections also play a pivotal role in Kronheimer and Mrowka's $SO(3)$ instanton homology for webs and its relation to the Four-Color Theorem \cite{km-tait}.

The authors did not see how to provide a general algebraic proof of Theorem \ref{thm:2}, but we believe it can be done. Our main purpose in this article is to exhibit how the congruence in Theorem \ref{thm:main} requires hardly any work, once the picture for the relevant moduli spaces is established.

Finally, it should be mentioned that although we refer to the invariant $\lambda(Y,E)$ as a `Casson invariant,' we are using Taubes' interpretation \cite{taubes} of Casson's invariant for integral homology 3-spheres, applied to non-trivial admissible bundles. \\

\vspace{.25cm}
\noindent {\bf{Outline.}} In Section \ref{sec:admissible} we review the notion of non-trivial admissibility and the suitable generalization which motivates the hypotheses of Theorem \ref{thm:2}. Sections \ref{sec:setup} and \ref{sec:klein4} provide the background for the main argument of Theorem \ref{thm:main}, which was sketched above and is presented concisely in Section \ref{sec:mainarg}. The issue of perturbations is ignored here, and then taken up in Section \ref{sec:hol}. In Section \ref{sec:thm2} we complete the proof of Theorem \ref{thm:2}. Finally, in Section \ref{sec:ex} we prove Proposition \ref{prop:main}, record a connected sum formula for the parity of $v_Y(x)$, and discuss double branched covers.\\

\vspace{.25cm}
\noindent {\bf{Acknowledgments.}} The authors would like to thank Danny Ruberman and Nikolai Saveliev for helpful discussions. The first author was supported by NSF grant DMS-1503100.\\

\vspace{.25cm}

\section{Non-trivial admissible bundles}\label{sec:admissible}

Here we briefly discuss Floer's non-trivial admissibility condition. A good reference for this material is \cite{braam-donaldson}. As in the introduction, we let $Y$ be a closed, oriented and connected 3-manifold. An $SO(3)$ bundle over $Y$ is {\emph{non-trivial admissible}} if its second Stiefel-Whitney class $x \in H^2(Y;\Z_2)$ satisfies the following three equivalent conditions, cf. \cite[Lemma 1.1]{braam-donaldson}:\\

\begin{itemize}
	\item The image of $x$ under $h:H^2(Y;\Z_2)\to \text{Hom}(H_2(Y;\Z),\Z_2)$ is non-zero.
	\item There is an {\emph{orientable}} surface $\Sigma \subset Y$ such that $\langle x, [\Sigma] \rangle \not\equiv 0$.
	\item The element $x\in  H^2(Y;\Z_2)$ has no torsion lifts to $H^2(Y;\Z)$.\\
\end{itemize}

\noindent One then defines a $U(2)$ bundle to be non-trivial admissible if its induced adjoint $SO(3)$ bundle is non-trivial admissible. The definition is motivated by the fact that a non-trivial admissible $U(2)$ bundle admits no reducible flat connections. This avoids complications in instanton Floer theory. Using that $h$ is surjective, and the fact that $SO(3)$ bundles over a 3-manifold are characterized by the second Stiefel-Whitney class, we count the number of non-trivial admissible $SO(3)$ bundles:
\[
	(2^{b_1(Y)}-1)2^{b_1(2)-b_1(Y)}.
\]
According to Theorem \ref{thm:main} and Poudel's result mentioned in the introduction, the parity of $v_Y(E)$ is the same for all non-trivial admissible bundles $E$. However, Theorem \ref{thm:2} indicates that the parity of $v_Y(E)$ is invariant under a larger collection of bundles. Such bundles are characterized by having a second Stiefel-Whitney class $x\in H^2(Y;\Z_2)$ that satisfies the following equivalent conditions:\\

\begin{itemize}
	\item The image of $x$ under $g:H^2(Y;\Z_2)\to \text{Hom}(\text{PD}(\text{ker}(\beta^1)),\Z_2)$ is non-zero.
	\item There is a surface $\Sigma \subset Y$ such that $\langle x, [\Sigma] \rangle \not\equiv 0$ and $\Sigma\cdot \Sigma \equiv 0 \in H_1(Y;\Z_2)$.
	\item The element $x\in  H^2(Y;\Z_2)$ has no order 2 lifts to $H^2(Y;\Z)$.
	\item The element $x\in  H^2(Y;\Z_2)$ is not the cup-square of an element from $H^1(Y;\Z_2)$.\\
\end{itemize}

\noindent Note that here $\Sigma$ is not necessarily orientable. Also, $\text{PD}:H^1(Y;\Z_2)\to H_2(Y;\Z_2)$ is the Poincar\'{e} duality isomorphism. These conditions are the natural extensions of the prior three conditions when one wants to treat $\Z$-summands and $\Z_{2^k}$-summands for $k>1$ the same. We note that the ring $H^\ast(Y;\Z_2)$ cannot see the difference between such summands. The map $g$ is surjective, so the number of $SO(3)$ bundles of this more general type is
\[
	(2^{k(Y)}-1)2^{b_1(2)-k(Y)}.
\]
The most basic example of such a bundle that is not non-trivial admissible is the non-trivial $SO(3)$ bundle over the lens space $L(4,1)$.\\
\vspace{.25cm}

\section{Configuration spaces and stabilizers}\label{sec:setup}

Fix a connection $A_0$ on $\det(E)$, and let $\CE$ be the space of connections $A$ on $E$ with determinant connection $\tr(A)=A_0$. Let $\GE$ be the gauge transformation group consisting of smooth unitary automorphisms of $E$ that are determinant 1. The configuration space is the quotient $\BE = \CE/\GE$. The non-trivial admissibility of $E$ implies that all points in $\BE$ are irreducible, meaning that the $\GE$-stabilizer of every connection in $\CE$ is as small as possible:
\[
    \text{Stab}_{\GE}(A) =\{\pm 1\}.
\]
The $U(2)$ bundle $E$ induces an $SO(3)$ bundle $\adE$, which may be defined as the sub-bundle of $\text{End}(E)$ consisting of trace-free, skew-hermitian endomorphisms. We let $\GadE$ denote the full $SO(3)$ gauge transformation group of $\adE$. 
Any $A\in \CE$ induces a connection $A_\ad\in\CadE$, and this induces a bijection between $\CE$ and $\CadE$. Indeed, any $U(2)$ connection $A$ on $E$ is uniquely determined by $\tr(A)$ on $\det(E)$ and $A_{\ad}$ on $\adE$. In contrast to the $U(2)$ case, we have
\[
    \text{Stab}_{\GadE}(A_\ad) \in \Big\{\{1\},\; \Z_2,\; V_4\Big\}.
\]
Indeed, the difference between the determinant 1 unitary gauge group and the $SO(3)$ gauge group is described by an action of $H^1(Y;\Z_2)$ on $\BE$ that gives $\BadE$ as its quotient space.
The action is as follows: $H^1(Y;\Z_2)$ parametrizes the isomorphism classes of flat complex line bundles (with connection) $\chi$ with holonomy $\{\pm 1\}$. Then $[\chi]$ acts on $[A]\in\BE$ by tensoring the bundle-with-connection $(E,A)$ with $\chi$. We then have the more precise statement that $\text{Stab}_{\GadE}(A_\ad)$ is naturally a subspace of $H^1(Y;\Z_2)$, with the constraint that
\[
     \dim_{\Z_2}\text{Stab}_{\GadE}(A_\ad) \in \{0,1,2\}.
\]
In summary, we see that even though any connection in $\BE$ is irreducible, its image in $\BadE$ may not be irreducible. Connections on $\su(E)$ with stabilizer $\Z_2$ are exactly those connections whose holonomy is contained in $O(2)$ and properly contains a Klein-4 group. Equivalently, these are connections that are compatible with a splitting
\[
    \su(E) = \lambda\oplus L
\]
where $\lambda$ is a non-trivial real line bundle and $L$ is an unoriented real 2-plane bundle, and for which the connection on $L$ is irreducible. Connections with stabilizer $V_4$ are those whose holonomy is also isomorphic to $V_4$. Equivalently, these are connections compatible with a splitting
\[
    \su(E) = \lambda_1\oplus \lambda_2\oplus \lambda_3
\]
into a sum of three non-trivial real line bundles. We write $\BadE$ as a disjoint union
\[
    \BadE = \BadE^{\ast} \cup \BadE^{\Z_2} \cup \BadE^{V_4}
\]
in terms of irreducible connections, connections with $\Z_2$-stabilizer, and those with $V_4$-stabilizer. We refer to this last set as {\emph{Klein-four connections}}.\\

\begin{remark}
    If the assumption of non-trivial admissibility is removed, three other kinds of stabilizers in the $SO(3)$-gauge group can occur: $SO(2)$, $O(2)$ and $SO(3)$.\\
\end{remark}

\section{Klein-four connections}\label{sec:klein4}

The subset of Klein-four connections, unlike the other two subsets of $\BadE$, is a finite, discrete set. As the elements are characterized by having holonomy $V_4$, a finite group, they must all be flat, as a simple continuity argument shows. Alternatively, each splitting $\su(E) = \lambda_1\oplus \lambda_2\oplus \lambda_3$ into non-trivial real line bundles supports a unique compatible connection, which of course must be flat. For a moment, let us consider the larger set
\[
    \mathcal{B}^{\geq V_4} = \left\{\begin{array}{c}\text{connections over $Y$ on any $SO(3)$} \\\text{bundle with holonomy inside a $V_4$}\end{array}\right\}\bigg/ \text{gauge}.
\]
Then $\mathcal{B}^{\geq V_4}$ is parametrized by $SO(3)$ bundles of the form $\lambda_1\oplus \lambda_2\oplus \lambda_3$ over $Y$. Noting that $w_1(\su(E))=0$, sending such a bundle to the triple $\{w_1(\lambda_1),w_1(\lambda_2),w_1(\lambda_3)\}$ sets up a bijection
\[
    \mathcal{B}^{\geq V_4} \;\; \overset{1:1}\longleftrightarrow \;\;  \left\{\begin{array}{c} \{a,b,c\} \subset H^1(Y;\Z_2)\\ \text{with  }\; a + b + c =0\end{array}\right\} =: V_Y.
\]
Yet another description of $\mathcal{B}^{\geq V_4}$ is as the set of homomorphisms $\text{Hom}(\pi_1(Y),V_4)$ modulo the action of $S_3=\text{Aut}(V_4)$. A simple counting argument shows that $\mathcal{B}^{\geq V_4}$ has cardinality
\[
    2^{b_1(2)-1} + \frac{1}{6}\left(4^{b_1(2)}+2\right).
\]
Now, the elements of $\mathcal{B}^{\geq V_4}$ that live on $\su(E)$ are the ones with
\[
    w_2(E) = w_2(\lambda_1\oplus \lambda_2\oplus \lambda_3) = a_1a_2 + a_2 a_3 + a_1a_3, \qquad a_i = w_1(\lambda_i).
\]
Thus we have the following bijection describing Klein-four connections on the bundle $\su(E)$:
\[
    \BadE^{V_4} \;\; \overset{1:1}\longleftrightarrow \;\;  \left\{\begin{array}{c} \{a,b,c\} \in V_Y \text{ with} \\ \; ab + bc + ac = w_2(E)\end{array}\right\}.
\]
We see now that $v_Y(E)=\left|\BadE^{V_4}\right|$, and the statement of the main theorem is the congruence
\begin{equation}
	\lambda(Y,E) \equiv 2^{b_1(2)-3}\cdot\left|\BadE^{V_4}\right| \mod 2^{b_1(2)-2}.\label{eq:cong2}
\end{equation}
\vspace{.20cm}

\section{The argument modulo perturbations}\label{sec:mainarg}

We now prove the theorem under the assumption that all moduli spaces to follow are nondegenerate, so that no perturbations are needed. The argument uses the most basic information we have from the $H^1(Y;\Z_2)$-action. Consider the moduli space of projectively flat connections on $E$:
\[
 \Mpflat_E :=  \left\{ [A]\in \BE: \; F_A = \frac{1}{2}F_{A_0}\cdot\text{id}_{E}\right\}.
\]
This is a finite set, and each of its points is irreducible. This moduli space is invariant under the $H^1(Y;\Z_2)$-action, and its quotient is the moduli space of flat connections on $\su(E)$:

\[
	\M_{\adE} := \left\{[B]\in \BadE: \; F_{B} =0 \right\}.
\]
\vspace{.02cm}

\noindent We need the following observation. An element $w\in H^1(Y;\Z_2)$ affects the relative mod 8 Floer grading $\text{gr}[A]$ of a connection $[A]\in \M_E$ by the formula (see \cite[Prop. 1.9]{braam-donaldson})
\[
    \text{gr}\left(w\cdot [A]\right) -\text{gr}[A]\;\equiv \;4\left( w_2(E)w + w^3 \right) \mod 8,
\]
so the $H^1(Y;\Z_2)$-action preserves the $\Z_2$-gradings. In particular, each $H^1(Y;\Z_2)$-orbit lies in a single $\Z_2$-grading. The proof is now completed by counting orbit sizes. Each connection in $\M_{\adE}$ with stabilizer at most $\Z_2$ gives an orbit of size either 
\[
    |H^1(Y;\Z_2)|=2^{b_1(2)} \quad \text{ or }\quad  |H^1(Y;\Z_2)/\Z_2|=2^{b_1(2)-1}
\]
lying upstairs in $\Mpflat_E$. Thus $2^{b_1(2)-1}$ divides the {\emph{signed}} count of $\Mpflat_E$, with the prior observation about gradings in mind. The remaining connections downstairs in $\M_{\su(E)}$ are Klein-four connections, and so in fact are given by the set $\mathcal{B}_{\su(E)}^{V_4}$. Each point in this set contributes an orbit of size 
\[
    |H^1(Y;\Z_2)/V_4|=2^{b_1(2)-2}
\]
upstairs in $\Mpflat_E$. Recalling that $\lambda(Y,E)$ is {\emph{half}} the signed count of points in $\M_E$, we recover the congruence (\ref{eq:cong2}), proving the theorem under the assumption of non-degeneracy.\\

\vspace{.30cm}

\section{Including holonomy perturbations}\label{sec:hol}

In general, the moduli space $\M_E$ is degenerate and we need to perturb the projectively flat equation to achieve the transversality we want. Henceforth we assume that our 3-manifold $Y$ is equipped with a Riemannian metric. The standard class of perturbations used are known as {\emph{holonomy perturbations}} \cite{herald, rs1}. The input for such a perturbation is an embedding $\Gamma=\{\gamma_k\}_{k=1}^m$ into $Y$ of solid tori $\gamma_k:S^1\times D^2\to Y$. We require that the embedded tori $\gamma_k$ have a common normal disk, meaning that the image of $\{1\}\times D^2$ under $\gamma_k$ is the same for all $k$. We also require that the images of the core loops $S^1\times\{0\}$ are disjoint away from the normal disk. Fix a trivialization of $\det(E)$ over the image of $\Gamma$, which is homotopically a wedge (bouquet) of circles. This allows us to consider the holonomy around the $\gamma_k$ as living in $SU(2)$. Let $f:SU(2)^m\to \R$ be a conjugation invariant function, i.e.,
\[
    f(g a_1 g^{-1}, \ldots , ga_mg^{-1}) = f(a_1,\ldots,a_m)
\]
for all $g\in SU(2)$. We also choose a smooth 2-form $\mu$ on $D^2$ with compact support in the interior and integral 1. From this data one constructs a holonomy perturbation $h$, given as follows:
\[
    h(A) = \int_{D^2} f\left(\text{Hol}_{\gamma_{1,z}}(A),\ldots,\text{Hol}_{\gamma_{m,z}}(A)\right)\mu(z).
\]
Here $\gamma_{k,z}$ is the loop $t\mapsto \gamma_k(t,z)$ in $Y$. Fixing only the data $\Gamma$, we define $\mathcal{H}_\Gamma$ to be the space of perturbations constructed as above. Each $h\in \mathcal{H}_\Gamma$ yields a well-defined function $h:\BE\to \R$.

One way to guarantee that the perturbation $h$ is $H^1(Y;\Z_2)$-equivariant is to require that each loop $\text{im}(\gamma_k)$ is zero as a class in $H_1(Y;\Z_2)$. We call such $\Gamma$ {\emph{mod-2 trivial}}, following \cite{rs1}, where this condition is introduced. We record their observation:\\

\begin{lemma}
    If $\Gamma$ is mod-2 tivial, then each $h\in \mathcal{H}_\Gamma$ is $H^1(Y;\Z_2)$-equivariant.\\
\end{lemma}

\noindent Now, the perturbed $U(2)$ moduli space $\M_E^h$ is the set of critical points of the perturbed Chern-Simons functional $\text{CS}+h$. Specifically, for a suitable normalization of $\text{CS}$, we obtain
\[
	\M_E^h = \left\{[A]\in \mathcal{B}_E: \;\; F_{A} - \frac{1}{2}F_{A_0}\cdot \text{id}_E+ \star\nabla h = 0 \right\}.
\]
If $\Gamma$ is mod-2 trivial, this perturbed moduli space inherits the $H^1(Y;\Z_2)$-action from $\mathcal{B}_E$, and its quotient space is the perturbed $SO(3)$ moduli space for $\su(E)$. We also record the following:\\

\begin{lemma}
Suppose $\Gamma$ is mod-2 trivial. For any $h\in \mathcal{H}_\Gamma$, Klein-four connections are unmoved in the $SO(3)$ moduli space. More precisely, we always have the relation
\[
    \M_{\su(E)}^h \cap \mathcal{B}_{\su(E)}^{V_4} = \mathcal{B}_{\su(E)}^{V_4}.
\]\label{lemma:unmoved}
\end{lemma}

\noindent As such perturbations are $H^1(Y;\Z_2)$-equivariant, a similar statement holds for the connections in the $U(2)$ moduli space $\M^h_E$ lying above Klein-four connections.

Our goal is to find a mod-2 trivial $\Gamma$ such that for small, generic $h\in\mathcal{H}_\Gamma$ the moduli space $\M_E^h$ is non-degenerate. Section 5 of \cite{rs1} shows that this can be achieved if $\Gamma$ is {\emph{abundant}} at each projectively flat $[A]\in \M_E$. We need to slightly generalize the definition of abundancy given in \cite{rs1}, which only considers stabilizers isomorphic to $\{1\}$ and $\Z_2$. To begin, note that $H^1(Y;A_\text{ad})$, the Zariski tangent space to $[A]$ in $\M_E$, carries an action by the stabilizer, denoted
\begin{equation}
	S_A:=  \text{Stab}_{H^1(Y;\Z_2)}[A] = \text{Stab}_{\G_{\su(E)}}(A_\text{ad}). \label{eq:stab5}
\end{equation}
We remark that the second equality in (\ref{eq:stab5}) is not true in general, and is contingent upon the non-trivial admissibility of $E$.
Recall that $S_A$ is one of $\{1\}$, $\Z_2$ or $V_4$. Now, decompose the tangent space into its $S_A$-invariant subspace $V_A$, and the $S_A$-equivariant orthogonal complement to $V_A$:\\
\[
	H^1(Y;A_\text{ad}) = V_A\oplus V_A^\perp.
\]
The space $V_A$ is the Zariski tangent space of $[A]$ internal to the stratum of $\M_E$ consisting of connection classes with stabilizer isomorphic to $S_A$. The complement $V_A^\perp$ is the Zariski normal bundle fiber in $\M_E$ at $[A]$ relative to the aforementioned stratum. For a vector space $W$ we write $\text{Sym}(W)$ for the space of symmetric bilinear forms on $W$. If $W$ has a linear $G$-action by some group $G$, we write $\text{Sym}(W)^G$ for the forms that are $G$-invariant.\\

\begin{defn}\label{defn:1}
A mod-2 trivial $\Gamma$ is {\emph{abundant}} at a projectively flat $[A]\in \M_E$ if there exist perturbations $\{h_i\}_{i=1}^n \subset \mathcal{H}_\Gamma$ and some $k$ such that $Dh_i(A)=0$ for $k+1\leq i\leq n$, and such that the following map that is defined from $\R^n$ to {\emph{$\text{Hom}(V_A,\R)\oplus \text{Sym}(V_A^\perp)^{S_A}$}} is surjective:
\begin{equation}
    (x_1,\ldots,x_n) \;\;\longmapsto \;\; \left( \;\sum_{i=1}^k Dh_i(A), \;\; \sum_{i=k+1}^n x_i\; {\emph{\text{Hess}}}\; h_i(A)\; \right).\label{eq:hess}
\end{equation}
\end{defn}

\vspace{.35cm}

\noindent Note that if $S_A$ is trivial, then $V_A$ accounts for the entire tangent space, and in particular $V_A^\perp=0$. Thus only the left-hand factor of the map (\ref{eq:hess}) is relevant. This is the condition of `first order abundancy,' and is sufficient to achieve non-degeneracy for small, generic perturbations when there are no other (lower) strata to consider. At the other extreme, when $S_A$ is isomorphic to $V_4$, we have $V_A=0$. In this case (\ref{eq:hess}) reduces to a condition purely of `second order abundancy.' 

If $S_A$ is isomorphic to $\Z_2$, then $V_A$ and $V_A^\perp$ are the $+1$ and $-1$ eigenspaces of the $\Z_2$-action, respectively, and are $V_+$ and $V_-$ in the notation of \cite{rs1}. In this case $\text{Sym}(V_A^\perp)^{S_A}$ is the same $\text{Sym}(V_-)$. Our choice of $\text{Sym}(V_A^\perp)^{S_A}$ in Defn. \ref{defn:1} is sufficient for the arguments of Section 5 in \cite{rs1} to go through in part because a generic element therein is non-degenerate, cf. the proof of Prop. 5.4 in \cite{rs1}. When $S_A$ is isomorphic to $\{1\}$ or $\Z_2$, our definition agrees with that of \cite{rs1}.

We are left with producing a mod-2 trivial $\Gamma$ which is abundant for all $[A]\in \M_E$. To this end, the work of Ruberman and Saveliev implies the following:\\

\begin{lemma}[{\cite[Prop. 5.2]{rs1}}]
    There exists a mod-2 trivial $\Gamma$ that is abundant for all connections in $\M_E$ that do not descend to $SO(3)$ Klein-four connections.\label{lemma:rubsav}\\
\end{lemma}

\noindent This allows us to focus on the situations in which $S_A$ is isomorphic to $V_4$, the case in which $A_\text{ad}$ is a Klein-four connection. We have the following facts, used in \cite[\S 5.5]{rs1}, stated informally:\\
\begin{itemize}
    \item If $\Gamma$ is abundant, and $\Gamma'$ is close to $\Gamma$, then $\Gamma'$ is abundant.
    \item If $\Gamma$ is abundant and $\Gamma\subset \Gamma'$, then $\Gamma'$ is abundant.
\end{itemize}
\vspace{.25cm}

\noindent In these situations, we are assuming that $\Gamma$ and $\Gamma'$ have the same fixed normal disk with basepoint. Now suppose we can show, for each $A$ with $S_A$ isomorphic to $V_4$, the existence of a mod-2 trivial $\Gamma$ abundant at $[A]$. Then it is straightforward to conclude, using these two facts and Lemma \ref{lemma:rubsav}, that there exists a mod-2 trivial $\Gamma'$ abundant at all $[A]\in \M_E$. Thus the following lemma completes the proof of Theoerem \ref{thm:main}:\\

\begin{lemma}
    There is an abundant mod-2 trivial $\Gamma$ for any $[A]\in \M_E$ that descends to an $SO(3)$ Klein-four connection. \\
\end{lemma}

\begin{proof}
	We follow the method used in \cite{rs1} of passing to a finite cover. Let $A$ be a projectively flat connection on $E$ with stabilizer $S_A$ isomorphic to $V_4$. The $SO(3)$ connection $A_\text{ad}$ is compatible with a splitting $\su(E)= \lambda_1\oplus \lambda_2\oplus \lambda_3$ in which the $\lambda_i$ are non-trivial and distinct real line bundles. The stabilizer $S_A$ is given explicitly by
\[
	S_A = \{0,a_1,a_2,a_3\} \subset H^1(Y;\Z_2), \qquad a_i = w_1(\lambda_i).
\] 
Here, $a_i$ corresponds to the gauge transformation of $\su(E)$ that simultaneously reflects $\lambda_{i+1}$ and $\lambda_{i+2}$, while fixing $\lambda_i$, where indices are taken mod 3. Define a homomorphism $\pi_1(Y)\to S_A$ by
\[
	\gamma\;\; \longmapsto \;\;  a_1(\gamma)a_1 + a_2(\gamma)a_2 + a_3(\gamma)a_3.
\]
Let $p:Y'\to Y$ be the covering space corresponding to this homomorphism. Under this covering $A_\text{ad}$ pulls back to a trivial connection, denoted $A'_\text{ad}$, cf. \cite[Lemma 5.6]{rs1}. In particular, each of $\lambda_i$ pulls back under $p$ to a trivial real line bundle $\lambda'_i$. Note that the covering transformation group of $Y'\to Y$ is the Klein-four group $S_A$.

It is known \cite[Prop. 67 \& Lemma 58]{herald} that there is some $\Gamma'$, a collection of embedded solid tori in $Y'$, that is abundant at the trivial connection $A'_\text{ad}$ in the following sense: there exist perturbations $\{ h_i\}_{i=1}^n \subset \mathcal{H}_{\Gamma'}$ such that the map from $\R^n$ to $\text{Sym}(H^1(Y';A'_\text{ad}))^{SO(3)}$ given by

\begin{equation}
    (x_1,\ldots,x_n) \;\;\longmapsto \;\;  \sum_{i=1}^n x_i\; \text{Hess}\; h_i(A'_\text{ad})\label{eq:abundancy2}
\end{equation}
is surjective. The appearance of the $SO(3)$ here is the gauge stabilizer of the connection $A_\text{ad}'$. Let $\Gamma$ be the image of $\Gamma'$ under $p$, slightly perturbed in $Y$ so that it is of the form described at the beginning of this section. By construction, $\Gamma$ is mod-2 trivial. Consider the following map:
\begin{equation}
        \text{Sym}(H^1(Y';A_\text{ad}'))^{SO(3)}\longrightarrow \text{Sym}(H^1(Y;A_\text{ad}))^{V_4}.\label{eq:sym}
\end{equation}
Here the $V_4$ refers to $S_A$. The map (\ref{eq:hess}) is the composition of (\ref{eq:abundancy2}) with (\ref{eq:sym}). Thus, to show abundancy of $\Gamma$ at $A$, it suffices to show that (\ref{eq:sym}) is surjective. The map (\ref{eq:sym}) is induced by the pull-back map:

\begin{equation}
\begin{CD}
{\scriptstyle{V_4}} \;\;@.  \rcirclearrowright H^1(Y;A_\text{ad})   @>{p^\ast}>> H^1(Y';A_\text{ad}') \lcirclearrowleft {\scriptstyle{SO(3)}}
\label{eq:pullback}
\end{CD}
\end{equation}
\vspace{.10cm}

\noindent This map is equivariant with respect to the indicated gauge stabilizer actions, upon considering $V_4$ as a subgroup of $SO(3)$. More precisely, $V_4$ refers to the $\mathcal{G}_{\su(E)}$-stabilizer of $A_\text{ad}$, while $SO(3)$ refers to the $\mathcal{G}_{p^\ast\su(E)}$-stabilizer of $A'_\text{ad}$.

To show that (\ref{eq:sym}) is surjective, consider the following two decompositions:
\begin{equation}
    H^1(Y;A_\text{ad}) = \bigoplus_{i=1}^3 H^1(Y;\lambda_i),\qquad H^1(Y';A'_\text{ad}) = H^1(Y';\R)\otimes\R^3.\label{eq:decomps}
\end{equation}
Implicit here is a trivialization for each $\lambda'_i$, and the $\R^3$ should be thought of as coming from the induced trivialization of $\lambda'_1\oplus\lambda_2'\oplus\lambda_3'$. The map (\ref{eq:pullback}) respects these decompositions. In the left-hand decomposition of (\ref{eq:decomps}), the $V_4$ action is as follows: $a_i$ acts as $-1$ on $H^1(Y;\lambda_{i+1})\oplus H^1(Y;\lambda_{i+2})$, and $+1$ on $H^1(Y;\lambda_i)$. In the tensor product appearing in (\ref{eq:decomps}), the $SO(3)$-action on $\R^3$ is standard, and is trivial on $H^1(Y';\R)$. From these descriptions, it is straightforward to verify that these decompositions induce identifications between the domain and codomain of (\ref{eq:sym}) with $\text{Sym}(H^1(Y';\R))$ and $\bigoplus_{i=1}^3\text{Sym}(H^1(Y;\lambda_i))$, respectively.
The map (\ref{eq:sym}) can then be seen as the map 
\begin{equation}
    \text{Sym}(H^1(Y';\R)) \longrightarrow \;\bigoplus_{i=1}^3\;\text{Sym}(H^1(Y;\lambda_i)),\label{eq:newsym}
\end{equation}
in which each of the three components is the map induced by pull-back, after trivializing $\lambda_i'$. Now, (\ref{eq:newsym}) is surjective because the three relevant pull-back maps are injective, and their three images pairwise intersect at $0$. This is evident from the decomposition
\[
    H^1(Y';\R) = H^1(Y;\R)\oplus H^1(Y;\lambda_1)\oplus H^1(Y;\lambda_2)\oplus H^1(Y;\lambda_3),
\]
which is induced by the covering transformation group $S_A$ acting on $H^1(Y';\R)$. This action should not to be confused with the gauge stabilizer action of $S_A$ on $H^1(Y;A_\text{ad})$ which was used above. The summand $H^1(Y;\R)$ is the invariant subspace under this action, while $H^1(Y;\lambda_i)$ is the complement of $H^1(Y;\R)$ inside the invariant subspace for the subgroup $\{0,a_i\}$.
\end{proof}

\vspace{.85cm}

\begin{remark}
For a discussion of some of the technical assumptions used here, see Section 5.6 of \cite{rs1}. For a detailed study of the abundancy of holonomy perturbations in the context of the equivariant Kuranishi method, see \cite{herald2}.
\end{remark}
\vspace{.45cm}

\section{Establishing the vanishing result}\label{sec:thm2}

Here we complete the proof of Theorem \ref{thm:2}. The remaining step is to use a realization result for the $\Z_2$-cohomology ring due to Turaev in conjunction with Corollary \ref{cor:main}. Recall that for a closed, oriented and connected 3-manifold we have the triple cup product form
\[
	u_Y:H^1(Y;\Z_2)\otimes H^1(Y;\Z_2)\otimes H^1(Y;\Z_2) \longrightarrow \Z_2,
\]
\[
	u_Y(a,b,c) = \left(a\cup b \cup c\right)[Y].
\]
The trilinear form $u_Y$ determines the $\Z_2$-cohomology ring of $Y$. It was originally proven by Postnikov that any symmetric trilinear form satisfying $u(a,a,b)=u(b,b,a)$ is realized by a closed, oriented and connected 3-manifold. Recall also that we have the linking form
\[
	L_Y:\text{Tor}H_1(Y;\Z)\otimes \text{Tor}H_1(Y;\Z) \longrightarrow \Q/\Z,
\]
which is a non-degenerate symmetric bilinear form. The linking form interacts with the $\Z_2$-cohomology ring in the following way. Let $\psi:\Z_2\to \Q/\Z$ be the injection defined by $\psi(k\;(\text{mod } 2)) = k/2$. Then for all $a,b\in H^1(Y;\Z_2)$ we have the relation
\begin{equation}
	\psi\left(u_Y(a,a,b)\right) = L_Y(a^\dagger,b^\dagger), \label{eq:star2}
\end{equation}
where for any $a\in H^1(Y;\Z_2)$ the element $a^\dagger\in \text{Tor}H_1(Y;\Z)$ is defined by the condition that $L_Y(a^\dagger,c) =\psi(a(c))$ for all $c\in \text{Tor}H_1(Y;\Z)$. Here we are of course identifying $H^1(Y;\Z_2)$ with $\text{Hom}(H_1(Y;\Z),\Z_2)$. An implication of Turaev's work is the following:\\

\begin{theorem}[\cite{turaev}]\label{thm:turaev}
	Let $H$ be a finitely generated abelian group, and let $u:{\emph{\text{Hom}}}(H,\Z_2)^{\otimes 3}\to \Z_2$ be a symmetric trilinear form. There exists a closed, orientable and connected 3-manifold $Y$ such that the pair $(H,u)$ is equivalent to $(H_1(Y;\Z),u_Y)$ if and only if there exists a non-degenerate symmetric bilinear form $L:{\emph{Tor}}H^{\otimes 2}\to \Q/\Z$ such that (\ref{eq:star2}) holds with $u_Y=u$ and $L_Y=L$.\\
\end{theorem}

\begin{proof}[Proof of Theorem \ref{thm:2}] Let $Y$ be such that $k(Y)\geq 4$, and suppose $x$ is not a cup-square. Equivalently, $x$ has no order 2 lift to $H^2(Y;\Z)$. We aim to show that $v_Y(x)\equiv 0$ (mod 2). We choose an isomorphism
\[
    H_1(Y;\Z) \simeq \bigoplus_{i=1}^4 A_i \oplus B
\]
where $A_i$ is an abelian group of the form $\Z_{2^k}$ for $k>1$ or a copy of $\Z$. Make these choices so that $x$ has a lift to $H^2(Y;\Z)$ with support in $A_1$, not of order 2, which can be done by our assumption on $x$. Recall that $\text{Tor}H_1(Y;\Z)$ is the torsion of $H^2(Y;\Z)$ by the Universal Coefficients Theorem. Now define $H$ by replacing the $A_i$ summands with copies of $\Z$:
\[
    H:=\bigoplus_{i=1}^4 A'_i \oplus B, \qquad A_i' := \Z
\]
With our identifications we have a natural isomorphism between $H^1(Y;\Z_2)$ and $\text{Hom}(H,\Z_2)$, and with this understood we set $u:=u_Y$. Also, noting that $\text{Tor} H$ is simply $\text{Tor} H_1(Y;\Z)$ with some summands possibly thrown away, we define $L$ to be the restriction of $L_Y$. With our identifications, the terms appearing in (\ref{eq:star2}) are unchanged. Thus Theorem \ref{thm:turaev} implies the existence of a closed, oriented and connected 3-manifold $Z$ with first homology and triple cup product form given by $(H,u)$. By our choices, $x$ has no torsion lifts, and is thus equal to $w_2(E)$ for a non-trivial admissible $U(2)$ bundle $E$ over $Z$. Now Poudel's result in the guise of Corollary \ref{cor:main} says $v_Z(x) \equiv 0$ (mod 2), since $b_1(Z)\geq 4$. Since the $\Z_2$-cohomology rings of $Y$ and $Z$ are the same, we then get $v_Y(x) \equiv 0$ (mod 2). The independence of $x$ as a choice having no order 2 lift to $H^2(Y;\Z)$ is established in much the same way as the vanishing.
\end{proof}

\vspace{.25cm}

\section{Examples and properties}\label{sec:ex}

\noindent In this section we prove Proposition \ref{prop:main}, which yields examples of $v_Y(x)$ (mod 2) for Seifert fibered spaces. We then produce a connected sum formula for the parity of $v_Y(x)$. Finally, we illustrate how to compute $v_Y(x)$ for double branched covers of links.\\
\vspace{.65cm}

\noindent {\bf{Seifert fibered spaces.}} Let $Y$ be a Seifert fibered 3-manifold over an oriented base orbifold, with Seifert invariants $(g,b,(\alpha_1,\beta_1),\ldots,(\alpha_r,\beta_r))$. Here $g$ is the genus of the base orbifold. The mod 2 cohomology ring of $Y$ is completely described in \cite{adhsz}.\\

\begin{lemma} Suppose $x\in H^2(Y;\Z_2)$ is not a square. If any of the $\alpha_i$ are even, or if all $\alpha_i$ are odd and $b+\sum \beta_i \equiv 1 \mod 2$, then $v_Y(x)=0$.\\
\end{lemma}

\begin{proof}
We begin with the following easily verified observation. In general, we have
\begin{equation}
    \{ a^2 : \; a\in H^1(Y;\Z_2)\} \subset \{ ab: \; a,b\in H^1(Y;\Z_2) \}.\label{eq:square}
\end{equation}
When these sets are equal, then $v_Y(x)=0$. For if the triple $\{a,b,a+b\}\in V_Y$ had $a^2+b^2+ab=x$, then $x$ would in fact be a square, contradiction. Now we appeal to \cite[Thm. 2.9]{adhsz}. 
When there is some even $\alpha_i$ (``case $n=0$'' in \cite{adhsz}), we easily check that these two sets in (\ref{eq:square}) are equal. This is particularly immediate when there is an $\alpha_i$ divisible by 4, and the mod 2 cohomology ring of $Y$ is isomorphic to that of a connect sum of some copies of $\R\mathbb{P}^3$ and some copies of $S^1\times S^2$. Finally, if all $\alpha_i$ are odd and $b+\sum \beta_i \equiv 1 \mod 2$, then the ring is isomorphic to that of a connect sum of $2g$ copies of $S^1\times S^2$, whence by the same reasoning $v_Y(x)=0$.
\end{proof}
\vspace{.30cm}

\begin{proof}[Proof of Proposition \ref{prop:main}] First, since $b_1(Y)$ is equal to either $2g$ or $2g+1$, $v_Y(x)$ is even by Corollary \ref{cor:main} unless $g=1$. By the above lemma, it remains to check that $v_Y(x)\equiv 1\mod 2$ when $g=1$ and all $\alpha_i$ are odd and $b+\sum \beta_i \equiv 0 \mod 2$. One can conclude from \cite{adhsz} that $H^1(Y;\Z_2)$ has a basis $a,b,c$ with $a^2=b^2=0$ and non-zero products $ab$, $bc$, $ac$, the three of which provide a basis for $H^2(Y;\Z_2)$. Depending on some divisibility conditions on the $\beta_i$, either $c^2=0$ or $c^2=ab$. The element $ac$, for one, is never a square, so we set $x=ac$. In either case we compute $v_Y(x)=1$.
\end{proof}
\vspace{.85cm}

\noindent {\bf{Connected sums.}} Now let $x$ be any element of $H^2(Y;\Z_2)$. Recall that $V_Y$ may be viewed as $\text{Hom}(\pi_1(Y),V_4)$ modulo the action of $S_3=\text{Aut}(V_4)$. As such, it makes sense to keep track of the $S_3$-stabilizers of the orbits. For a set $X$ with $S_3$-action we define the triple $\check{v}(X)=(\check{v}_1,\check{v}_2,\check{v}_3)$ where $\check{v}_1$, $\check{v}_2$, $\check{v}_3$ are the numbers of orbits with stabilizers of orders 1, 2, 6, respectively. For two such sets $X_1$ and $X_2$ with $S_3$-actions we have
\[
	\check{v}({X_1\times_{S_3} X_2}) = \check{v}({X_1})\times \check{v}({X_2})
\]
where we define the product $\times$ on triples as follows:

\[
	\check{v}\times \check{u} := (6\check{v}_1\check{u}_1+3\check{v}_1\check{u}_2 + 3 \check{v}_2\check{u}_1+\check{v}_1\check{u}_3+\check{v}_3\check{u}_1 + \check{v}_2\check{u}_2,\;\;\; \check{v}_2\check{u}_2 + \check{v}_2\check{u}_3 + \check{v}_3\check{u}_2,\;\;\; \check{v}_3\check{u}_3).
\]
Define the norm of a triple to be the $L^1$-norm: $|\check{v}| = \check{v}_1 + \check{v}_2 + \check{v}_3$. 
Write $\check{v}_Y(x)$ for the triple $\check{v}(X)$, with $X$ the subset of $\text{Hom}(\pi_1(Y),V_4)$ that lives on an $SO(3)$ bundle $E$ with $x=w_2(E)$. Thus $X/S_3$ is the subset of $\{a,b,c\}\in V_Y$ such that $ab+bc+ac=x$. With our new notation, we have
\[
	v_Y(x) = |\check{v}_Y(x)|.
\]
Now, given $x_i\in H^2(Y_i;\Z_2)$ it is easy to verify the connect sum relation
\[
	v_{Y_1\# Y_2}(x_1+x_2)  = |\check{v}_{Y_1}(x_1)\times \check{v}_{Y_2}(x_2)|.
\]
Note also that if $x$ is not a cup-square, then $\check{v}_Y(x)$ has the form
\[
    \check{v}_Y(x) = (\check{v}_1,0,0).
\]
In general, the third entry $\check{v}_3$ is equal to 1 if and only if $x=0$, and is otherwise $0$. Also, the second entry $\check{v}_2$ is the number of non-trivial cup-square-roots of $x$:
\[
    \check{v}_2 = |\{a\in H^1(Y;\Z_2):\; a\neq 0, \; a^2 = x\}|, \;\;\text{ where } \check{v}_Y(x) = (\check{v}_1,\check{v}_2,\check{v}_3).
\]
In particular, the sum $\check{v}_2+\check{v}_3$ is either zero or the cardinality of the kernel of the Bockstein map $H^1(Y;\Z_2)\to H^2(Y;\Z_2)$, which is by definition $2^{k(Y)}$. Putting these observations together, and using our freedom to choose $x$ that is not a square (below choose $x_2=0$), we compute the following:\\

\begin{prop}
Suppose $x_i\in H^2(Y_i;\Z_2)$ and that $x_1$ is not a cup-square. Then
\[
    v_{Y_1\# Y_2}(x_1+x_2) \equiv \begin{cases}
    v_{Y_1}(x_1)\mod 2, & k(Y_2)=0\\
    0 \mod 2, & \text{otherwise}
    \end{cases}
\] \label{prop:2}
\end{prop}

\noindent In particular, we recover the fact (mod 2) that $v_Y(x)$ is stable under connect summing with $\R\mathbb{P}^3$. More generally, these statements clearly hold when the decompositions are only algebraic, instead of geometric: for example, if there is a decomposition $H^1(Y;\Z_2)=A\oplus B$ where $A\cup B=0$ and $B$ has an element of order 4 or $\infty$, then $v_Y(x)\equiv 0$ (mod 2) for any $x$ not a cup-square.\\

\vspace{.85cm}

\noindent {\textbf{Double branched covers.}} The above Seifert-fibered examples for genus $g=0$ are double branched covers of Montesinos links, but in all of those cases $v_Y(x)$ vanishes (mod 2) for non-squares $x$. Here we compute a non-vanishing example in which $Y$ is a double branched cover $\Sigma(L)$ of a link $L$ in $S^3$. First, we describe the $\Z_2$ cohomology rings of such manifolds. Let $L$ be a link with components $L_1,\ldots L_n$, and let $S_i$ be a Seifert surface for $L_i$. Then $S_i$ lifts to a closed surface $F_i$ in the branched cover $\Sigma(L)$. Write $a_i\in H^1(\Sigma(L);\Z_2)$ for the Poincar\'{e} dual of $[F_i]$.\\

\begin{prop}
Let $L$ be an $n$ component link. The vector space $H^1(\Sigma(L);\Z_2)$ has dimension $n-1$, and it is generated by the $n$ classes $a_i$ subject to the one relation
\begin{equation}
    a_1 + \cdots + a_n = 0.\label{eq:reln}
\end{equation}
The triple cup product form on $H^1(\Sigma(L);\Z_2)$ is determined by the values
\[
    (a_i\cup a_j \cup a_k)[\Sigma(L)] \equiv  \begin{cases}
    \sum_{\ell\neq i}\text{{\emph{lk}}}(L_i,L_\ell) \mod 2, & i=j=k,\\
    \text{{\emph{lk}}}(L_i,L_k) \mod 2, & i=j\neq k,\\
    0 \mod 2, & i,j,k\text{  distinct}
    \end{cases}
\]\label{prop:cupring}
\end{prop}
\vspace{.20cm}

\noindent This proposition is proved for two-component links in \cite[Prop. 9.2]{ps}, and the proof easily generalizes. We sketch the argument. To begin, we mention that $H_1(\Sigma(L);\Z_2)$ is in bijection with the subsets of $\{1,\ldots,n\}$ of {\emph{even}} cardinality: 
\begin{equation}
    H_1(\Sigma(L);\Z_2) \;\; \overset{1:1}\longleftrightarrow \;\;  \left\{\begin{array}{c} S \subset \{1,\ldots,n\}\\  |S| \equiv 0 \mod 2\end{array}\right\}.\label{eq:bijectionsets}
\end{equation}
The bijection goes as follows. Given such a subset, pair off elements. For the pair $\{i,j\}$, draw an arc in $S^3$ between components $L_i$ and $L_j$, otherwise missing $L$. Lift the arcs to a union of loops in $\Sigma(L)$ to obtain a class in $H_1(\Sigma(L);\Z_2)$. Now, assume the $F_i$ are transverse to one another. Then it is not hard to see, when $i\neq j$, that $F_i\cap F_j$ is mod 2 homologous to 
\[
    \text{lk}(L_i,L_j)\cdot \{i,j\}
\]
where we view $\{i,j\}$ as an element of $H_1(\Sigma(L);\Z_2)$ via the above bijection. Upon taking Poincar\'{e} duals, this yields the proposition. We note that addition on the subsets appearing on the right side of (\ref{eq:bijectionsets}) is the symmetric difference of sets.

\begin{figure}[t]
\centering
\includegraphics[scale=.50]{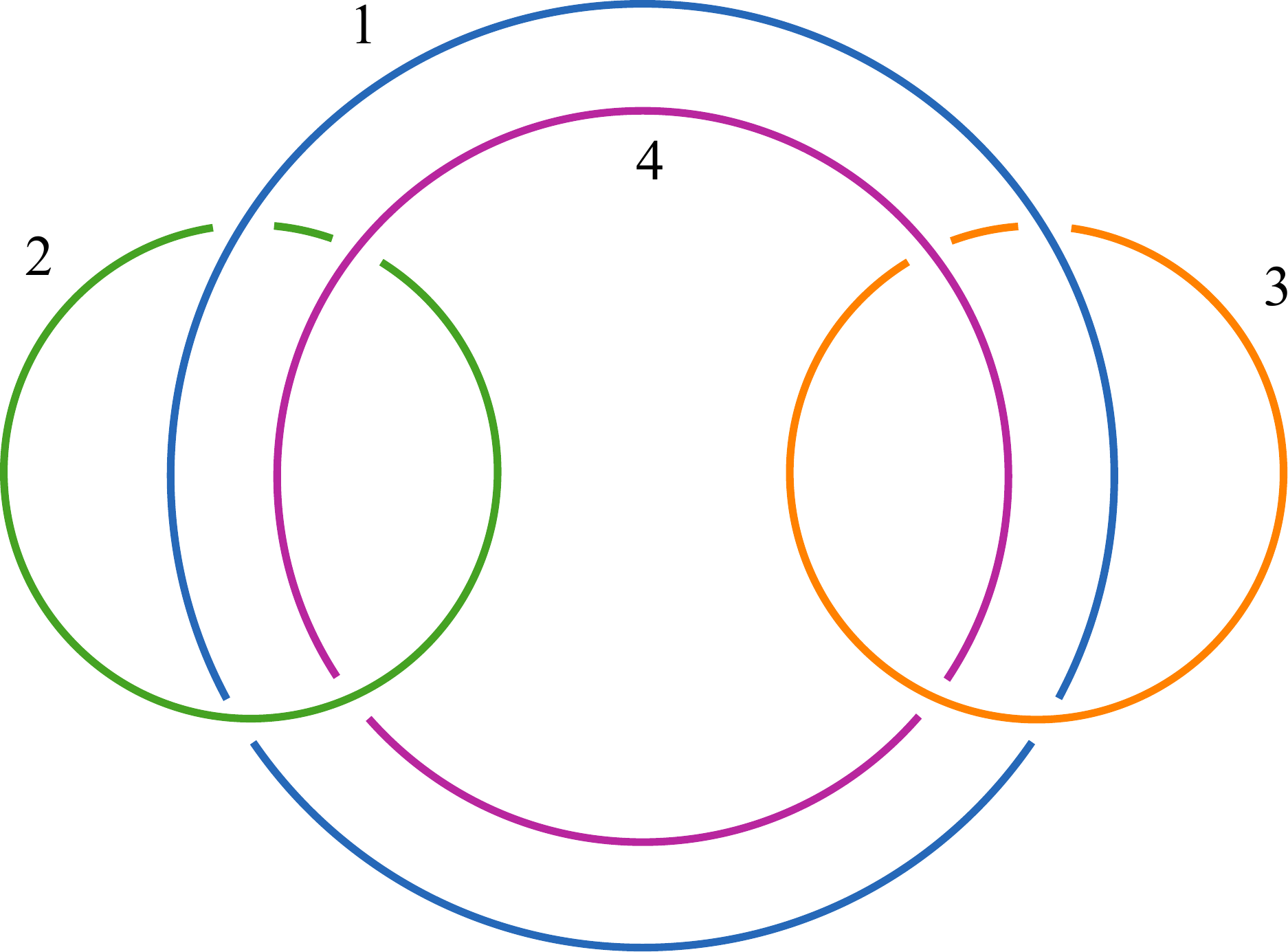}
\caption[]{{\small{
The link $L=$L8n8 with its four components labelled by $\{1,2,3,4\}$. This link has determinant zero and thus its branched double cover supports non-trivial admissible bundles.
}}}\label{fig:l8n8}
\end{figure}

Let $Y=\Sigma(L)$, and let $\func$ be the function from $V_Y$ to $H_1(Y;\Z_2)$ that sends a flat Klein-four connection class to the Poincar\'{e} dual of its second Stiefel-Whitney class:
\[
    \func\{a,b,c\} = \text{PD}(ab + bc + ac).
\]
Let $L$ be the four component link L8n8 depicted in Figure \ref{fig:l8n8}, and let $a_i$ be the classes described in Proposition \ref{prop:cupring} for $L$, so that $a_i$ is dual to the lifted Seifert surface of $L_i$. In particular, $a_1,a_2,a_3$ form a basis for $H^1(Y;\Z_2)$. For illustration, using Proposition \ref{prop:cupring} we compute:
\[
    \text{PD}(a_1^2) = \text{PD}\left(a_1 (a_2 + a_3 + a_4)\right) = \sum_{i=2}^4 \text{lk}(L_1,L_i)\cdot\{1,i\} = \{1,2\} + \{1,3\} = \{2,3\}.
\]
The bijection (\ref{eq:bijectionsets}) is implicit in our notation, aligning subsets of $\{1,2,3,4\}$ of even size with elements of $H_1(Y;\Z_2)$. We then compute $f$ on all fifteen of the Klein-four connection classes in $V_Y$:

\begin{align*}
 \func\{0,0,0\} &= 0, & \func\{a_1,a_2,a_1+a_2\} &= {\color{WildStrawberry}\{3,4\}},\\
 \func\{a_1,a_1,0\} &= \{2,3\}, & \func\{a_1,a_3,a_1+a_3\} &= {\color{WildStrawberry}\{2,4\}},\\
 \func\{a_2,a_2,0\} &= \{1,4\}, & \func\{a_2,a_3,a_2+a_3\} &= 0,\\
 \func\{a_3,a_3,0\} &= \{1,4\}, & \func\{a_1,a_2+a_3,a_1+a_2+a_3\} &= 0,\\
 \func\{a_1+a_2,a_1+a_2,0\} &= \{1,2,3,4\}, & \func\{a_2,a_1+a_3,a_1+a_2+a_3\} &= {\color{WildStrawberry}\{1,3\}},\\
 \func\{a_1+a_3,a_1+a_3,0\} &= \{1,2,3,4\}, & \func\{a_3,a_1+a_2,a_1+a_2+a_3\} &= {\color{WildStrawberry}\{1,2\}},\\
 \func\{a_2+a_3,a_2+a_3,0\} &= 0, & \func\{a_1+a_2,a_1+a_3,a_2+a_3\} &= 0,\\
 \func\{a_1+a_2+a_3,a_1+a_2+a_3,0\} &= \{2,3\}. & & \\
\end{align*}

\noindent We find that the cup-squares form a 2-dimensional subspace of $H^2(Y;\Z_2)$, appearing as the outputs of the left-hand column. Thus $k(Y)= 1$. We have four non-squares, appearing as the non-zero entries (in red) in the right-hand column. Each has one Klein-four class, and so $v_Y(x)\equiv 1$ (mod 2) when $x$ is not a cup-square. The link $L$ has determinant zero, i.e. $b_1(Y)>0$, so $Y$ has a non-trivial admissible $U(2)$ bundle $E$. By Theorem \ref{thm:main} we conclude
\[
    \lambda(Y,E) \equiv 1 \mod 2.
\]
Proposition \ref{prop:cupring} similarly computes the parity of $2^{4-n}\lambda(Y,E)$, when $\det(L)=0$, from only knowing the mod 2 linking matrix of $L$.
\vspace{.35cm}

\bibliography{main.bbl}
\bibliographystyle{alpha}

\vspace{.85cm}

\footnotesize

  \textsc{Department of Mathematics, Brandeis University,
    Waltham, MA}\par\nopagebreak
  \textit{E-mail address:}\;\texttt{cscaduto@brandeis.edu}

\vspace{.35cm}

    \textsc{Department of Mathematics, University of California,
    Los Angeles, CA}\par\nopagebreak
  \textit{E-mail address:}\;\texttt{mstoffregen@math.ucla.edu}

\end{document}